\documentclass[a4paper,oneside,11pt]{article} 

\usepackage[utf8]{inputenc} 
\usepackage[T1]{fontenc} 
\usepackage{textcomp} 
\usepackage[english]{babel} 
\usepackage{indentfirst}
\usepackage[autolanguage]{numprint} 
\usepackage{hyperref} 
\hypersetup{colorlinks=true,linkcolor=blue,citecolor=red,urlcolor=blue}
\usepackage{graphicx} 
\usepackage{verbatim} 
\usepackage{makeidx} 
\usepackage[refpage]{nomencl} 
\usepackage{enumitem} 
\usepackage{tikz} 
\usetikzlibrary{matrix,arrows,patterns}
\usepackage{multicol}
\usepackage{fullpage}
\usepackage{comment}
\usepackage{stmaryrd}
\usepackage{mathrsfs}
\usepackage{ulem}

\usepackage{amsmath,amssymb,amsthm,amsfonts} 
\usepackage[all]{xy} 

\numberwithin{equation}{section}

\newcounter{Main}

\theoremstyle{plain} 

\newtheorem{MainThm}[Main]{Theorem} 

\swapnumbers 
\theoremstyle{definition} 
\newtheorem{Def}{Definition}[section] 
\newtheorem{Def,Thm}[Def]{Definition and theorem} 
\newtheorem{Def,Prop}[Def]{Proposition-definition} 

\theoremstyle{plain} 
\newtheorem{prop}[Def]{Proposition} 
\newtheorem{lem}[Def]{Lemma} 
\newtheorem{Thintro}{Theorem}
\newtheorem{cor}[Def]{Corollary} 
\theoremstyle{remark} 

\usepackage{xcolor}
\usepackage[all]{xy}

\newcommand{\V}{\mathrm{u}_{2a}}
\newcommand{\U}{\mathrm{u}_a}

\title{Cohomology of special unitary groups and congruence subgroups} 
\author{Claudio Bravo\footnote{Instituto de Matem\'aticas, Universidad de Talca, Talca, Chile.
Email address: claudio.bravo@utalca.cl.}} 

\date{}

\begin{document} 

\maketitle

\begin{abstract}
We prove a homotopy invariance result for the first cohomology group of the special unitary group 
$\mathrm{SU}_3(F[t])$ with coefficients in irreducible representations of $\mathrm{PGL}_2(F)$. 
The main theorem establishes that this cohomology is naturally isomorphic to the corresponding 
cohomology of $\mathrm{PGL}_2(F)$.
\end{abstract}

\section{Introduction}\label{section introduction}

The fundamental theorem of algebraic K-theory states that for each regular ring $R$ there are natural isomorphisms between the $i$-th K-theory groups $K_i(R[t]) \cong K_i(R)$, for all $i\geq 0$ (cf. \cite[Theorem 8, \S 6, Ch. 8]{Ktheory}).
The first K-theory group of a ring $R$ can be explicitly described as follows: Let $\mathrm{GL}(R)$ be the direct limit $\varinjlim \mathrm{GL}_n(R)$ given by the inclusions
$\mathrm{GL}_n(R) \hookrightarrow \mathrm{GL}_{n+1}(R)$, $A \mapsto \begin{pmatrix}
A & 0 \\
0 & 1 
\end{pmatrix}.$
The group $K_1(R)$ is isomorphic to the abelianization $\mathrm{GL}(R)^{\mathrm{ab}}$ of $\mathrm{GL}(R)$.
Since $\mathrm{GL}(R)^{\mathrm{ab}} \cong H_1\big( \mathrm{GL}(R), \mathbb{Z} \big)$, the fundamental theorem of K-theory implies that:
$$H_1 \big( \mathrm{GL}(R), \mathbb{Z} \big) \cong H_1 \big( \mathrm{GL}(R[t]), \mathbb{Z} \big).$$
The existence of homotopy invariance results in K-theory motivates research on unstable versions for the homology of algebraic groups such as $\mathrm{GL}_n$ and $\mathrm{SL}_n$.
To state the homotopy invariance question for homology in its most general form, let us write $\mathcal{G}$ for a linear (algebraic) group scheme $\mathcal{G}$ over a ring $R$, i.e., a ``family of algebraic groups'' parametrized by the points of the curve $\mathrm{Spec}(R)$.
The homotopy invariance question in homology asks under what conditions the canonical map $\mathcal{G}(R) \to \mathcal{G}(R[t])$ induces isomorphisms on the homology groups:
$$H_*\big(\mathcal{G}(R),M\big) \xrightarrow{\cong} H_*\big(\mathcal{G}(R[t]),M\big),$$
for suitable coefficient modules $M$.


In ninety-six, Knudson studied in \cite{Knudson1} the case of $\mathrm{SL}_2$ over fields $F$, where $\mathrm{char}(F)=0$, proving that 
the inclusion $\mathrm{SL}_2(F) \hookrightarrow \mathrm{SL}_2(F[t])$ induces isomorphisms between the corresponding integral 
homology groups. See \cite[Theorem 3.1]{Knudson1}. 
One year later, Knudson generalized in \cite{Knudson2} his aforementioned results to the higher-rank group $\mathrm{SL}_n$ over arbitrary infinite fields, removing the earlier characteristic zero hypothesis in $F$.

Based on \cite{Margaux}, Wendt obtains in \cite{Wendt} one of the strongest results on the homology invariance question in homology. Indeed, he proves the following:

\begin{Thintro}\cite[Theorem 1.1]{Wendt}
If $F$ is an infinite field and $\mathcal{G}$ is a connected reductive smooth group defined over $F$, then the inclusion $F \hookrightarrow F[t]$ 
induces isomorphisms between the integral homology groups 
$$H_*\big(\mathcal{G}(F),\mathbb{Z}\big) \xrightarrow{\cong} H_*\big(\mathcal{G}(F[t]),\mathbb{Z}\big),$$
whenever the order of fundamental group of $\mathcal{G}$ is invertible in $F$.
\end{Thintro}


Several positive answers to the homotopy invariance question for other arithmetic groups, different from $\mathcal{G}(F[t])$, are given in works as either \cite[Ch. 4]{KnudsonBook} or \cite{Hutchinson, Knusdonelementarygroups}. 
Assuming that $F$ is a finite field, some positive answers for the homotopy invariance question are given in \cite{Soulé,Knudsonfinitefields}.

Going to the case of arbitrary regular rings, it follows from \cite{KrMc} that homotopy invariance does not hold for $H_1$ over $\mathcal{G}(R[t])$, when $\mathrm{rk}(\mathcal{G})=1$ and $R$ is an integral domain which is not a field.
In the same context, homotopy invariance fails for $H_2$  whenever $\mathrm{rk}(\mathcal{G})=2$, as discussed in \cite{Wendt2}.
Thus, the assumption that $R=F$ is a field seems crucial to obtain positive answers to the homotopy invariance question in homology.

Note that all previous results are specific to algebraic groups defined over fields $F$, which can seen as isotrivial algebraic groups schemes.
Thus, it is natural to seek for extensions to group schemes defined over projective algebraic $F$-curves, such as $\mathbb{P}^1_F$, but not on $F$.
Note that, if $\mathcal{G}$ is such a group, i.e. a $\mathbb{P}^1_F$-group scheme, then, the (abstract) group $\mathcal{G}(\mathbb{A}_F^{1})$ plays the role of $\mathcal{G}(F[t])$ in the isotrivial case.

Usually, algebraic groups (and also algebraic group schemes) are studied in families according to suitable properties that satisfy them.
Indeed, a group $\mathcal{G}$ defined over $K$ is said split if it contains a maximal torus that is $K$-isomorphic to $\mathbb{G}_m^r=\mathrm{GL}_1^r$, for certain $r \in \mathbb{Z}$.
We say also that $\mathcal{G}$ is quasi-split if it contains a Borel $K$-subgroup. Every split group is quasi-split, but not conversely (See the example of $\mathrm{SU}_3$ in \S \ref{section curves}).

Recall that each split semisimple group has a $\mathbb{Z}$-model (called its Chevalley model) according to \cite[Exp.~XXV~1.3]{SGA3-3}.
In particular, they are all isotrivial, since we can define an $F$-model by extension of scalars.
Thus, in order to study the homotopy invariance question in the context of algebraic group schemes, 
we are led to consider groups outside of this category.
The subsequent natural family of groups to consider is the category of semi-simple simply connected non-split quasi-split groups.
In this context, the first progress was described in \cite{bravohominv} for the special unitary $\mathbb{P}^1_F$-group scheme $\mathrm{SU}_3=\mathrm{SU}_{3,\mathbb{P}^1_F}$, defined in detail in \S \ref{section curves}. More specifically, for this group scheme, we have:

\begin{Thintro}\cite[Theorem 2.1]{bravohominv}\label{teo1}
Let $F$ be an infinite field with $\mathrm{char}(F) \neq 2$.
There exists an injective (and natural) homomorphism $\iota: \mathrm{PGL}_2(F) \hookrightarrow \mathrm{SU}_3(F[t])$, which induces isomorphisms:
$$\iota_*: H_* \big( \mathrm{PGL}_2(F), \mathbb{Z} \big) \xrightarrow{\cong}  H_* \big( \mathrm{SU}_3(F[t]), \mathbb{Z} \big).$$
\end{Thintro}

It is not hard to see that $\mathrm{PGL}_2(F)$ is isomorphic to the group $\mathrm{SU}_3(\mathbb{P}^1_F)$ of $\mathbb{P}^1_F$-points of $\mathrm{SU}_3$. So, Theorem \ref{teo1} can be rephrased by saying that the natural inclusion $\mathrm{SU}_3(\mathbb{P}^1_F) \hookrightarrow \mathrm{SU}_3(\mathbb{A}^1_F)$ induces isomorphism between the involved integral homology groups. \\

Note that the previous results address the question of homotopy invariance for the coefficient module $\mathbb{Z}$ with the trivial action.
In what follows, we focus on the case of non-trivial coefficients.
Specifically, in \cite{knudsontwistedcoh}, Knudson studies the homotopy invariance of the first cohomology group of $\mathrm{SL}_n(F[t])$ with coefficients in an irreducible rational representation $V$ of $\mathrm{SL}_n(F)$.
In this setting, $\mathrm{SL}_n(F[t])$ acts on $V$ via the evaluation homomorphism $\mathrm{SL}_n(F[t]) \to \mathrm{SL}_n(F)$.
He obtains the following result:

\begin{Thintro}\cite[Theorem 5.2]{knudsontwistedcoh}\label{teo2}
Let $F$ is a field with $\mathrm{char}(F)=0$, and let us denote by $\mathrm{Ad}$ the adjoint representation of $\mathrm{SL}_n(F)$.
The first cohomology group $H^1\big(\mathrm{SL}_n(F[t]),V\big)$ satisfies:
\begin{equation}\label{eq 1}
H^1\big(\mathrm{SL}_n(F[t]),V\big)=  \begin{cases}
H^1\big(\mathrm{SL}_n(F),V\big), & \text{if } V\neq \mathrm{Ad}, \\
H^1\big(\mathrm{SL}_n(F),V\big) \oplus F^{\infty} & \text{if } V=\mathrm{Ad}, n=2, \\
H^1\big(\mathrm{SL}_n(F),V\big)\oplus F, & \text{if } V= \mathrm{Ad}, n \geq 3.
\end{cases} 
\end{equation}
\end{Thintro}

In order to prove the previous result, Knudson takes two approaches. The first one is to use the spectral sequence associated to the action of $\mathrm{SL}_n(F[t])$ on the Bruhat-Tits building associated to $\mathrm{SL}_n$ at $F(\!(t^{-1})\!)$. 
The second approach is to use the Hochschild-Serre spectral sequence associated with the group extension:
$$ 1 \to K \to \mathrm{SL}_n(F[t]) \xrightarrow{t=0}  \mathrm{SL}_n(F) \to 1,$$
where $K$ is a principal congruence subgroup of $\mathrm{SL}_n(F[t])$. \\

The main goal of this article is to study the homotopy invariance of the first cohomology group of $\mathrm{SU}_3(F[t])$ over non-trivial modules. 
In order to introduce our main results, let us write the following definitions.
Let $A,B$ be two $G$-modules and let $\mathrm{Hom}_{G}(A,B)$ be the group of $G$-invariant homomorphisms $A\to B$.
Let $B(F) \cong F^{*} \rtimes F$ be the group of upper triangular matrices in $\mathrm{PGL}_2(F)$.
Then $B(F)$ acts on the ideal $J:=\sqrt{t} F[\sqrt{t}]$ via $(a,b) \cdot f=af$, for $a\in F^{*}$, $b\in F$ and $f \in J$.
In particular, for any $\mathrm{PGL}_2(F)$-module $M$, the group $\mathrm{Hom}_{B(F)}\big(J, M\big)$ is well defined.
This group measures how far is $H^1\big( \mathrm{SU}_3(F[t]) , M \big)$ from being $H^1\big( \mathrm{PGL}_2(F), M \big)$, as the next result shows:

\begin{MainThm}\label{main teo 0}
Let $\mathrm{SU}_3$ be the $\mathbb{P}^1_F$-group scheme described above (see \S \ref{section curves} for details) and write $ \mathrm{SU}_3(F[t])=\mathrm{SU}_3(\mathbb{A}^1_F)$.
For any field $F$ with $\mathrm{char}(F)\neq 2$ and any $\mathrm{PGL}_2(F)$-module $M$, we have:
$$H^1\big( \mathrm{SU}_3(F[t]) , M \big) \cong H^1\big( \mathrm{PGL}_2(F), M \big) \oplus \mathrm{Hom}_{B(F)}\big(J, M\big).$$
\end{MainThm}

By using the previous result, in \S \ref{section cohomology} we concentrate on describing the first cohomology group $H^1\big( \mathrm{SU}_3(F[t]), V \big)$, where $V$ is an irreducible representation of $\mathrm{PGL}_2(F)$ (or equivalently, an irreducible $\mathrm{SU}_3(F[t])$-representation with a trivial action of a suitable congruence subgroup, according to Lemma \ref{lemma equiv rep}).
In this way, we accomplish both, extending Theorem \ref{teo1} to non-trivial modules and providing an analog of Theorem \ref{teo2} for the non-split quasi-split group $\mathrm{SU}_3$, via next result:

\begin{MainThm}\label{main teo 1}
Let $F$ be a field with $\mathrm{char}(F)=0$ and let $V$ be an irreducible representation of $\mathrm{PGL}_2(F)$.
In the notations of Theorem \ref{main teo 0}, we have:
$$H^1\big( \mathrm{SU}_3(F[t]), V \big) \cong H^1(\mathrm{PGL}_2(F),V).$$
\end{MainThm}

In contrast with Knudson's result for $\mathrm{SL}_n$ (Theorem \ref{teo2}), 
where the adjoint representation gives rise to an additional term, Theorem \ref{main teo 1} shows that in the unitary case no such additional contribution arises.\\

As we state in Lemma \ref{lemma 1}, there is a split exact sequence of the form:
$$1 \to \Gamma(t) \to \mathrm{SU}_3(F[t]) \to \mathrm{PGL}_2(F) \to 1 ,$$
where $\Gamma(t)$ is a principal congruence subgroup of $\Gamma:=\mathrm{SU}_3(F[t])$.
In analogy to the second approach of Knudson introduced above, our method focuses on describing and using the Hochschild-Serre spectral sequence associated the aforementioned exact sequence.
As we prove in Lemma \ref{lemma 2}, for each $\mathrm{PGL}_2(F)$-module $M$, we have:
$$
 H^1\big( \Gamma, M \big) \cong H^1\big( \mathrm{PGL}_2(F), M \big) \oplus \mathrm{Hom}_{\mathrm{PGL}_2(F)}\big( \Gamma(t)^{\mathrm{ab}}, M\big),
$$
so we reduce to explicitly compute the abelianization $\Gamma(t)^{\mathrm{ab}}$ of $\Gamma(t)$ and the corresponding homomorphism module.
Indeed, in \S \ref{section abelianization} we study $\Gamma(t)^{\mathrm{ab}}$ together with $\mathrm{Hom}_{\mathrm{PGL}_2(F)}\big( \Gamma(t)^{\mathrm{ab}}, M\big)$ due to an amalgamated product describing $\Gamma(t)$ (See Eq.\eqref{eq amal prod}). In particular, in \S \ref{section abelianization} we prove Theorem \ref{main teo 0}.
Then, in order to prove Theorem \ref{main teo 1}, in \S \ref{section cohomology} we concentrate on the case where $M$ is an irreducible $\mathrm{PGL}_2$-representation.


\section{Algebraic curves and definition of \texorpdfstring{$\mathrm{SU_{3,\mathbb{P}^1_F}}$}{SU3}}\label{section curves}

Assume that $\mathrm{char}(F)\neq 2$ and let us consider the $2:1$ (ramified) cover $\psi: \mathcal{C}=\mathbb{P}^1_F \to \mathcal{D}=\mathbb{P}^1_F$ given by $z\mapsto z^2$.
This cover corresponds to the quadratic extension field $L=F(\sqrt{t})$ over $K=F(t)$.
Let $R$ be a subring of $K$ such that $\mathrm{Quot}(R)=K$, and let $S\subset L$ be its integral closure in $L$.
At any affine subset $\mathrm{Spec}(R)\subset \mathcal{C}$, we denote by $\mathcal{G}_R$ the special unitary group-scheme defined from the $R$-hermitian form:
\begin{equation}\label{eq h_R}
h_R:S^3\to R, \quad h_R(x,y,z):=x\bar{z}+ y \bar{y} + z \bar{x},
\end{equation}
where $\overline{(\cdot)}$ denotes the non-trivial element in $\mathrm{Gal}(L/K)$.
Since the curve $\mathcal{C}$ can be covered by affine subsets $\mathrm{Spec}(R_i)$ with affine intersection, the groups $\mathcal{G}_{R_i}$ can be glued in order to define the $\mathcal{C}$-group scheme $\mathcal{G}=\mathrm{SU}_{3,\mathcal{C}}$.

In the sequel, we write $\mathcal{G}_E$ to denote an algebraic group defined over a field $E$.
The generic fiber $\mathcal{G}_K$ of $\mathcal{G}$ is the special unitary group $\mathrm{SU}_{3,K}$ consisting in matrices in $\mathrm{SL}_{3,L}$ preserving the hermitian form $h_K$.
Such a group, defined over $K$, is quasi-split semisimple and simply connected.
One maximal $K$-split torus $\mathcal{S}_K$ in $\mathcal{G}_K$ consists in the group of diagonal matrices:
$$\mathcal{S}_K :=\left \lbrace 
\begin{pmatrix}
s & 0 & 0 \\
0 & 1 & 0 \\
0 & 0 & s^{-1}
\end{pmatrix}
\Big\vert s\in \mathbb{G}_{m,K}=\mathrm{GL}_{1,K}\right \rbrace.$$
In particular, the (split) $K$-rank of $\mathcal{G}_K$ is $1$.
The centralizer $\mathcal{T}_K$ of $\mathcal{S}_K$ in $\mathcal{G}_K$ is a $K$-maximal torus of $\mathcal{G}_K$.
This group can be described as:
$$\mathcal{T}_K :=\left \lbrace
\begin{pmatrix}
\lambda & 0 & 0 \\
0 & \bar{\lambda} \lambda^{-1} & 0 \\
0 & 0 & \bar{\lambda}^{-1}
\end{pmatrix} \Big\vert \lambda\in R_{L/K}(\mathbb{G}_{m,L}) \right \rbrace,$$
where $R_{L/K}(\mathbb{G}_{m,L})$ is a Weil restriction of $\mathbb{G}_{m,L}$.
Note that $\mathcal{T}_K$ splits over $L$, but it fails to decompose over $K$.
More explicitly, note that $\mathcal{T}_L:=\mathcal{T}_K \otimes_{K} L \cong \mathbb{G}_{m,L}^2$, however $\mathcal{T}_K \not\cong \mathbb{G}_{m,K}^2$. 
Recall that isotrivial groups decomposes at finite extensions $F'/F$.
In particular, since $L$ does not have the form $F'(t)$, for some finite extension $F'/F$, the group scheme $\mathcal{G}$ is non-isotrivial, i.e., $\mathcal{G}$ does not have an $F$-model.

Now, at any closed point $P$ of $\mathcal{C}$ that fails to decompose at $\mathcal{D}$, the $K_P$-group $\mathcal{G}_{K_P}=\mathcal{G}_K \otimes_{K} K_P$ is quasi-split and it splits at the quadratic extension $L_P=L \otimes_K K_P$.
If $P$ decomposes at $\mathcal{D}$, then $\mathcal{G}_{K_P}=\mathrm{SL}_{3,K_P}$.

In the sequel, we consider the (abstract) group $\Gamma:=\mathrm{SU}_3(F[t])$ of $R=F[t]$-points of $\mathcal{G}$.
This group can be represented as the group of matrices in $\mathrm{SL}_3(F[\sqrt{t}])$ preserving the form $h_{R}$ in Eq. \eqref{eq h_R}.

\section{Congruence subgroups and reduction to \texorpdfstring{$\mathrm{PGL}_2$}{PGL2}}\label{section cong and red}

This section and \S \ref{section abelianization} are devoted to proving Theorem \ref{main teo 0}.
Indeed, let $\mathrm{ev}_0: \mathrm{SL}_3(F[\sqrt{t}]) \to \mathrm{SL}_3(F)$ be the group homomorphism induced by the evaluation of $t$ at $0$.
We denote by $\Gamma(t)$ the principal congruence subgroup of $\Gamma=\mathrm{SU}_3(F[t])$ defined by $\Gamma(t):= \ker(\mathrm{ev}_0) \cap \Gamma$.

\begin{lem}\label{lemma 1}
There exists a split exact sequence of the form:
\begin{equation}\label{eq cong sub}
1 \to \Gamma(t) \to \Gamma \xrightarrow{\pi} \mathrm{PGL}_2(F) \to 1. 
\end{equation}
\end{lem} 

\begin{proof}
Let $q: F^3\to F$, where $q(x,y,z)=2xz+y^2$, be the quadratic form on $F$ defined by the restriction of $h$ to $F$, and let $\mathrm{SO}_3=\mathrm{SO}(q)$ be the special orthogonal group defined from $q$.
It follows from \cite[Lemma 6.1]{bravohominv} that $\Gamma/\Gamma(t)\cong \mathrm{SO}_3(F)$.
Moreover, it follows from \cite[Ch. II, \S 9, (3)]{Dieudonnée} that $\mathrm{SO}_3(F) \cong \mathrm{PGL}_2(F)$.
Since $\mathrm{SO}_3(F)=\Gamma \cap \mathrm{SL}_3(F)$, the result follows.
\end{proof}

Note that, given a $\mathrm{PGL}_2(F)$-module $M$, we can endow $M$ with a $\Gamma$-module structure via $g\cdot m=\pi(g) \cdot m$, where $\pi$ is defined in Eq. \ref{eq cong sub} and $(m,g)\in M \times \Gamma$.
This action is evidently trivial for $\Gamma(t) \subseteq \Gamma$.
Conversely, any $\Gamma$-module $M$ with a trivial action of $\Gamma(t)$ is a $\mathrm{PGL}_2(F)$-module.
We say that a $G$-module $M$ is a $G$-representation when it is a finite dimensional vector space with a $G$-action by linear maps.
The next result follows from the previous discussion.

\begin{lem}\label{lemma equiv rep}
There is a bijection between the set of $\Gamma$-modules (resp. $\Gamma$-representations) with a trivial action of $\Gamma(t)$ and the set of $\mathrm{PGL}_2(F)$-modules (resp. $\mathrm{PGL}_2(F)$-representations). \qed
\end{lem}

In the sequel, we focus on the description of the cohomology group $H^1(\Gamma,M)$.
Before that, we briefly recall some basic fact on group cohomology.
Indeed, let $G$ be an abstract group and let $M$ be an (abelian) $G$-module.
A $1$-cocycle is a map $f:G \to M$ satisfying $f(gh) = f(g) + g \cdot f(h),$ for all $g,h \in G$.
A $1$-cocycle is called a $1$-coboundary if there exists $m \in M$ such that $f(g)= g \cdot m - m$, for all $g \in G$. 
The first cohomology group $H^1(G,M)$ of $G$ with coefficients in $M$ is the quotient of the group of $1$-cocylces by the group of $1$-coboundaries, defined above.
Equivalently, $H^1(G,M)$ classifies crossed homomorphisms from $G$ to $M$ up to equivalence.
Next result is an standard fact on the cohomology of group extensions, which we prove here for the sake of completeness.

\begin{lem}\label{lemma coh is h}
Let $1\to H \to G \to Q \to 1$ be a short exact sequence and let $M$ be a $G$-module where $H$ acts trivially.
Then:
$$H^0\big( Q, H^1(H,M)\big) \cong \mathrm{Hom}_Q\big(H^{\mathrm{ab}}, M \big) .$$
\end{lem}

\begin{proof}
Since $H$ acts trivially on $M$, we have $
H^1(H,M) \;\cong\; \mathrm{Hom}(H,M).$ Moreover, since $M$ is abelian, we have $H^1(H,M) \;\cong\; \mathrm{Hom}(H^{\mathrm{ab}},M).$
The $Q$-action on $H$ by conjugation descends to an action on $H^{\mathrm{ab}}$, and hence $\mathrm{Hom}(H^{\mathrm{ab}},M)$ becomes a $Q$-module in the usual way. Taking $Q$-invariants we obtain $H^0\big( Q,H^1(H,M) \big) \;\cong\; \mathrm{Hom}_Q \big( H^{\mathrm{ab}},M \big),$
which completes the proof.
\end{proof}

Now, we return to the description of the first cohomology group of $\Gamma$.

\begin{lem}\label{lemma 2}
For each $\mathrm{PGL}_2(F)$-module $M$, we have:
\begin{equation}\label{eq H1}
 H^1\big( \Gamma, M \big) \cong H^1\big( \mathrm{PGL}_2(F), M \big) \oplus \mathrm{Hom}_{\mathrm{PGL}_2(F)}\big( \Gamma(t)^{\mathrm{ab}}, M\big).
\end{equation}
\end{lem}

\begin{proof}
Since $\Gamma(t)$ acts trivially on $M$ and the exact sequence in Eq. \eqref{eq cong sub} splits, we have that the map $d_2^{0,1}: E_2^{0,1} \to E_2^{2,0}$ vanishes.
Thus, we have:
$$ H^1\big( \Gamma, M \big) \cong E_2^{1,0} \oplus E_2^{0,1} \cong H^1\big( \mathrm{PGL}_2(F), M \big) \oplus H^0\big( \mathrm{PGL}_2(F), H^1\big(\Gamma(t),M\big) \big).$$
Again, since $\Gamma(t)$ acts trivially on $M$, it follows from Lemma \ref{lemma coh is h} that the cohomology group $H^0\big( \mathrm{PGL}_2(F), H^1\big(\Gamma(t),M\big) \big)$ is isomorphic to the group $ \mathrm{Hom}_{\mathrm{PGL}_2(F)}\big(\Gamma(t)^{\mathrm{ab}}, M\big)$, whence the result follows.
\end{proof}

\section{On the abelianization of congruence subgroups}\label{section abelianization}

Let $\mathrm{N}, \mathrm{Tr}: L \to K$ be the norm and the trace maps defined by the quadratic extension $L/K$.
Let $H(L,K)$ and $H(L,K)^0$ be the sets:
\begin{align*}
H(L,K) &:= \left\{ (u,v) \in L^2 \big\vert \, \mathrm{N}(u) + \mathrm{Tr}(v) = 0\right\},\\
H(L,K)^0 &:= \left\{ v \in L \big\vert \, \mathrm{Tr}(v) = 0\right\}.
\end{align*}
These sets parameterize the root subgroups $\mathcal{U}_a(K)$ and $\mathcal{U}_{2a}(K)$ of $\mathcal{G}(K)$ through the maps $\U : H(L,K) \to \mathcal{U}_a(K) \subset\mathcal{G}(K)$ and $\V : H(L,K)^0 \to \mathcal{U}_{2a}(K) \subset \mathcal{G}(K)$, respectively defined by
\begin{equation*}
\U (u,v)= \begin{pmatrix} 1 & -\bar{u} & v \\ 0 & 1 & u \\ 0 & 0 & 1\end{pmatrix}, \quad \V(v) = \begin{pmatrix} 1 & 0 & v \\ 0 & 1 & 0 \\ 0 & 0 & 1\end{pmatrix},
\end{equation*}
where $(u,v) \in H(L,K)$ and $v \in H(L,K)^0$.
Let $J$ be the ideal $J=\sqrt{t}F[\sqrt{t}]$ and let us write $H(L,K)_J:=H(L,K)\cap (J \cap J)$.
We denote by $U_J$ the unipotent group
$$U_J=\lbrace \mathrm{u}_a(x,y) \vert (x,y) \in H(L,K)_J \rbrace.$$
It follows from \cite[Cor. 6.6]{bravohominv} that congruence subgroup $\Gamma(t)$ decomposes as a free product of the form:
\begin{equation}\label{eq amal prod}
\Gamma(t)={\ast}_{\mathbb{P}^1(F)} U_J.
\end{equation}
Recall that $\mathbb{P}^1(F)= \mathrm{PGL}_2(F)/ B(F)$, where $B$ is the group of upper triangular matrices in $\mathrm{PGL}_2$.
Then, it follows from Eq. \eqref{eq amal prod} that:
\begin{equation}\label{eq abs}
\Gamma(t)^{\mathrm{ab}} = \mathrm{Ind}_{B(F)}^{\mathrm{PGL}_2(F)} U_J^\mathrm{ab} = \bigoplus_{\mathbb{P}^1(F)} U_J^\mathrm{ab}.
\end{equation}
In order to study $\Gamma(t)^{\mathrm{ab}}$ we describe $U_J^{\mathrm{ab}}$ in the following lemma.

\begin{lem}\label{lemma Uj ab}
We have $U_J^{\mathrm{ab}}\cong J =\sqrt{t}F[\sqrt{t}]$.
\end{lem}

\begin{proof}
It follows from the matricial representation of $U_J$ that
it is a group with multiplication rule:
\begin{equation}\label{eq u}
\U(x,y)\U(x',y')=\U(x+x',\,y+y'+x\bar{x}').
\end{equation}
From this expression, it is easy to see that the commutator $[\U(u,v), \U(x,y)]$ equals $\V(0,u\bar{x}-\bar{u}x)$.
Note that the expression $u\bar{x}'-x'\bar{u}$ always lies in $J$. Hence $[U_J,U_J]\subseteq U_J^0$, where $U_J^{0}:=\lbrace \mathrm{u}_{2a}(y) \vert y \in H(L,K)^{0} \cap J \rbrace$.
On the other hand, let $\V(z)\in U_J^{0}$.
By definition of $J$, we can write $z=\sqrt{t}p(t)$, where $p(t) \in F[t]$.
Therefore: 
$$\V(z)=\big[\U\big(p(t)/2,-\mathrm{N}(p(t))/2\big), \U(-\sqrt{t},-t/2)\big],$$
which implies that $U_J^{0} \subseteq [U_J, U_J]$.
Now, let $f: U_J \to J$ be the map defined by $f\big(\U(x,y)\big)=x$. 
It follows from \eqref{eq u} that $f$ is a group homomorphism.
Moreover, since $\U(x,-N(x)/2)=x$, the map $f$ is surjective.
Since $\mathrm{ker}(f)=U_J^{0}$, we conclude that $U_{J}^{\mathrm{ab}}=U_J/U_J^{0} \cong J$.
\end{proof}

\begin{cor}
One has $\Gamma(t)^{\mathrm{ab}} \cong \bigoplus_{\mathbb{P}^1(F)} J$.
\end{cor}

Now, we are able to prove Theorem \ref{main teo 0}, which describes the first cohomology group of $\Gamma=\mathrm{SU}_3(F[t])$ with coefficient in an arbitrary $\mathrm{PGL}_2(F)$-module $M$.


\begin{proof}[Proof of Theorem \ref{main teo 0}]
It follows from Shapiro's lemma that
\begin{equation}\label{eq shapiro}
H^0\big( B(F) , H^1(J,M)\big) \cong H^0\Big( \mathrm{PGL}_2(F) , \mathrm{Coind}_{B(F)}^{\mathrm{PGL}_2(F)}   H^1(J,M)\Big).
\end{equation}
Since $J \cong U_J \subset \Gamma(t)$ acts trivially on $M$, Eq. \eqref{eq shapiro} together with Lemma \ref{lemma coh is h} implies that:
$$
\mathrm{Hom}_{B(F)}\big( J ,M \big) \cong \Big( \prod_{\mathbb{P}^1(F)} H^1(J,M) \Big)^{\mathrm{PGL}_2(F)} \cong \mathrm{Hom}_{\mathrm{PGL}_2(F)}\Big( \bigoplus_{\mathbb{P}^1(F)} J,M \Big).
$$
Since $\Gamma(t)^{\mathrm{ab}}\cong \bigoplus_{\mathbb{P}^1(F)} J$, we conclude that $\mathrm{Hom}_{B(F)}\big( J ,M \big)$ is isomorphic to $\mathrm{Hom}_{\mathrm{PGL}_2(F)}\big( \Gamma(t)^{\mathrm{ab}},M \big)$. Therefore, the result follows from Lemma \ref{lemma 2}.
\end{proof}

\section{Cohomology over rational representations}\label{section cohomology}

Let $M=V$ be a $\mathrm{PGL}_2(F)$-representation, i.e., assume that $M=V$ is a finite dimensional $F$-vector space with a linear action of $\mathrm{PGL}_2(F)$.
Assuming $\mathrm{char}(F)=0$, the next result describes $\mathrm{Hom}_{B(F)}\big(J, V\big)$.
This proposition, together with Theorem \ref{main teo 0}, completes the proof of Theorem \ref{main teo 1}.

\begin{prop}\label{prop final}
Assume that $\mathrm{char}(F)=0$ and let $V$ be a rational irreducible representation of $\mathrm{PGL}_2(F)$.
Then, we have $\mathrm{Hom}_{B(F)}\big(J, V\big) =\lbrace 0 \rbrace$.
\end{prop}

\begin{proof}
Let us recall that $B(F)$ decomposes as follows:
$$1 \to U(F) \to B(F) \to T(F) \to 1 ,$$
where $U(F)\cong F$ is the group of unipotent upper triangular matrices in $\mathrm{PGL}_2(F)$ and $T(F) \cong F^{*}$ is the group of diagonal matrices in $\mathrm{PGL}_2(F)$.
The Hochschild–Serre spectral sequence on the $0$-th row applied to $M=H^1 (J,V)$ implies that:
$$ H^{0}\Big(B(F), H^1 \big( J,V \big) \Big) = H^0 \Big( T(F), H^0\big( U(F), H^1(J,V) \big) \Big).$$
Hence, Lemma \ref{lemma coh is h} implies that
$$\mathrm{Hom}_{B(F)}(J^{\mathrm{ab}}, V) \cong H^0 \Big( T(F), \mathrm{Hom}_{U(F)}\big(J^{\mathrm{ab}},V\big) \Big).$$
Since $J$ is an abelian group, we have:
\begin{equation}\label{eq homology and homs}
\mathrm{Hom}_{B(F)}(J, V) \cong H^0 \Big( T(F), \mathrm{Hom}_{U(F)}\big(J,V\big) \Big).
\end{equation}
Moreover, since $U(F)$ acts trivially on $J$:
\begin{equation}\label{eq hom inv}
\mathrm{Hom}_{U(F)}\big(J,V\big) = \mathrm{Hom}\big(J, V^{U(F)} \big),  
\end{equation}
where $V^{U(F)}$ is the subspace of $V$ of $U(F)$-invariant elements.

Now, let $V$ be a rational irreducible representation of $\mathrm{PGL}_2(F)$.
It follows from \cite[\S 11, pag. 150]{fultonharris} that $V=\mathrm{Sym}^{2n} V_0$, where $n \in \mathbb{Z}_{\geq 0}$ and $V_0=F^2$ is the standard representation of $\mathrm{PGL}_2(F)$ (or $\mathrm{SL}_2(F)$).
Then $V^{U(F)}=V_{2n}$ is the highest weight space of $V$.
Hence, we get from Eqs. \eqref{eq homology and homs} and \eqref{eq hom inv} that:
\begin{equation}\label{eq homjv}
\mathrm{Hom}_{B(F)}(J, V) \cong H^0 \Big( T(F), \mathrm{Hom}\big( J, V_{2n} \big) \Big) \cong \mathrm{Hom}_{T(F)}\Big(J,V_{2n}\Big).
\end{equation}
Recall that $T(F) \cong F^{*}$ acts on $J\cong U(J)$ with a weight $1$, while it acts on $V_{2n}$ with a weight $2n$.
So, any $\phi \in \mathrm{Hom}_{T(F)}(J,V_{2n})$ satisfies $\phi(t \cdot x)= t \cdot \phi(x) = t^{2n} \phi(x)$, for all $t \in F^{*}$ and $x \in J$.
Hence, for $\alpha, \beta \in F$, we have:
$$ (\alpha+\beta)^{2n} \phi(x) = \alpha^{2n} \phi(x) + \beta^{2n} \phi(x).$$
Therefore $\phi(x)=0$, for all $x \in J$, since $\mathrm{char}(F)=0$. Hence, no non-trivial $T(F)$-equivariant map $\phi$ can exist.
We conclude that $\mathrm{Hom}_{T(F)}(J,V_{2n})=\lbrace 0 \rbrace$, whence $\mathrm{Hom}_{B(F)}\big(J, V\big) =\lbrace 0 \rbrace$.
This completes the proof of Prop. \ref{prop final} and, in turn, of Theorem \ref{main teo 1}.
\end{proof}

\small{
\section*{Acknowledgements}
This research was partially supported by ANID through Fondecyt Initiation Grant No. 11260422.
}


\bibliographystyle{plain}
\bibliography{refs.bib}

\end{document}